\documentclass[11pt]{article}

\usepackage[a4paper,margin=2.9cm]{geometry}
\usepackage[english]{babel}
\usepackage{parskip}
\usepackage{amssymb}
\usepackage{amsmath}
\usepackage{amsthm}
\usepackage{amssymb}
\usepackage{enumerate}
\usepackage{graphicx}
\usepackage{bbm} 
\usepackage{mathrsfs} 
\usepackage{subcaption}
\usepackage{verbatim}
\usepackage{hyperref}
\usepackage[percent]{overpic}

\makeatletter
\renewcommand\section{\@startsection{section}{1}{\z@}%
                                  {-3.5ex \@plus -1ex \@minus -.2ex}%
                                  {2.3ex \@plus.2ex}%
                                  {\normalfont\Large\bfseries}}
\makeatother

\newtheoremstyle{examplestyle}
  {4mm}
  {4mm}
  {\slshape}
  {0pt}
  {\bfseries}
  {\newline}
  {0mm}
  {}

\theoremstyle{examplestyle}
  \newtheorem{Theorem}{Theorem}[section]              

  \newtheorem{Lemma}[Theorem]{Lemma}
  
 \newtheorem*{Defs*}{Definitions}

\newtheorem{Conjecture}[Theorem]{Conjecture}
\newtheorem{Claim}[Theorem]{Claim}

\title{The Bollob\'as--Eldridge--Catlin conjecture for even girth at least $10$}

\author{Wouter Cames van Batenburg\thanks{Department of Mathematics, Radboud University Nijmegen, Postbus 9010, 6500 GL Nijmegen, The Netherlands. \href{mailto:w.camesvanbatenburg@math.ru.nl}{\nolinkurl{w.camesvanbatenburg@math.ru.nl}}, \href{mailto:ross.kang@gmail.com}{\nolinkurl{ross.kang@gmail.com}}.} \and Ross J. Kang\footnotemark[1]}

\begin{document}

\maketitle

\begin{abstract}
Two graphs $G_1$ and $G_2$ on $n$ vertices are said to \textit{pack} if there exist injective mappings of their vertex sets into $[n]$ such that the images of their edge sets are disjoint. A longstanding conjecture due to Bollob\'as and Eldridge and, independently, Catlin, asserts that, if $(\Delta(G_1)+1) (\Delta(G_2)+1) \le n+1$, then $G_1$ and $G_2$ pack. We consider the validity of this assertion under the additional assumptions that neither $G_1$ nor $G_2$ contain a $4$-, $6$- or $8$-cycle, and that $\Delta(G_1)$ or $\Delta(G_2)$ is large enough ($\ge 940060$).
\end{abstract}

\section{Introduction}

Two (simple) graphs $G_1$ and $G_2$ on $n$ vertices are said to {\em pack} if there exist injective mappings of their vertex sets into $[n] = \{1,\dots,n\}$ so that their edge sets have disjoint images. Equivalently, they pack if $G_1$ is a subgraph of the complement of $G_2$.

We let $\Delta_1$ and $\Delta_2$ denote the maximum degrees of $G_1$ and $G_2$, respectively.
The following, which posits a natural sufficient condition for $G_1$ and $G_2$ to pack in terms of $\Delta_1$ and $\Delta_2$, is a central combinatorial problem posed in the 1970s~\cite{BoEl78,Cat74,Cat76,SaSp78}.

\begin{Conjecture}[Bollob\'as and Eldridge~\cite{BoEl78} and Catlin~\cite{Cat76}]
If $G_1$ and $G_2$ are graphs on $n$ vertices with respective maximum degrees $\Delta_1$ and $\Delta_2$ such that $(\Delta_1 + 1) (\Delta_2 +1) \le n+1$, then they pack. 
\end{Conjecture}
If true, the statement would be sharp and would significantly generalise a celebrated result of Hajnal and Szemer\'edi~\cite{HaSz70}.
Sauer and Spencer~\cite{SaSp78} showed that $2 \Delta_1 \Delta_2 < n$ is a sufficient condition for $G_1$ and $G_2$ to pack, which is seen to be sharp when one of the graphs is a perfect matching.
Thus far the Bollob\'as--Eldridge--Catlin (BEC) conjecture has been confirmed in the following special cases: $\Delta_1 = 2$~\cite{AiBr93}; $\Delta_1 = 3$ and $n$ sufficiently large~\cite{CSS03}; $G_1$ bipartite and $n$ sufficiently large~\cite{Csa07}; and $G_1$ $d$-degenerate, $\Delta_1 \ge 40 d$ and $\Delta_2 \ge 215$~\cite{BKN08}.
In previous work~\cite{CK16}, we confirmed the BEC conjecture for maximum codegree of $G_1$ less than $t$ and $\Delta_1>17t\Delta_2$.

We would also like to highlight the following three results that can be considered approximate forms of the BEC conjecture.
(a) The BEC condition is sufficient for $G_1$ and $G_2$ to admit a `near packing' in that the subgraph induced by the intersection of their images has maximum degree at most $1$~\cite{Eaton00}.
(b) The condition $(\Delta_1 + 1) (\Delta_2 +1) \le 3n/5+1$ is sufficient for $G_1$ and $G_2$ to pack, provided that $\Delta_1,\Delta_2\ge 300$~\cite{KKY08}.
(c) If $G_2$ is chosen as a binomial random graph of parameters $n$ and $p$ such that $np$ in place of $\Delta_2$ satisfies the BEC condition, then $G_1$ and $G_2$ pack with probability tending to $1$ as $n\to\infty$~\cite{BJS14+}.


In this paper, we confirm the BEC conjecture for every pair of graphs neither of which contains a $4$-, $6$- or $8$-cycle as a subgraph --- i.e.~both of which have even girth at least $10$ --- provided at least one of the graphs has large enough maximum degree.

For the rest of the paper we always assume without loss of generality that $\Delta_1 \ge \Delta_2$.

\begin{Theorem}\label{th:newtheorem}
If $G_1$ and $G_2$ are graphs on $n$ vertices with respective maximum degrees $\Delta_1$ and $\Delta_2$ such that $(\Delta_1 + 1) (\Delta_2 +1) \le n+1$, then they pack provided neither contains a $4$-, $6$- or $8$-cycle and either $\Delta_1 \ge 940060$ or $\Delta_1\ge \Delta_2 \ge 27620$. 
\end{Theorem}

An important ingredient of the proof is a special case ($t=2$) of our previous result in~\cite{CK16}.

\begin{Theorem}[Corollary~1.5 in~\cite{CK16}]\label{th:oldtheorem}
If $G_1$ and $G_2$ are graphs on $n$ vertices with respective maximum degrees $\Delta_1$ and $\Delta_2$ such that $(\Delta_1+1)(\Delta_2+1) \le n+1$, then they pack provided $G_1$ contains no $4$-cycle and $\Delta_1 > 34 \Delta_2$.
\end{Theorem}

Thus, for Theorem~\ref{th:newtheorem}, we may restrict our attention to the case where $\Delta_1$ and $\Delta_2$ are relatively close to each other, i.e. $\Delta_2 \le \Delta_1 \le 34 \Delta_2$.
Central to our earlier work~\cite{CK16} was a lemma of Corr\'adi~\cite{Cor69}; the same is true here, but the application is more involved as we shall see in Section~\ref{sec:engine}.
We also use the `near packing' result~\cite{Eaton00}. 

We have made little effort to optimise the boundary constants $940060$ and $27620$. These constants partly depend on the constant $34$ in Theorem~\ref{th:oldtheorem}.

The structure of the paper is as follows. In the next section, we introduce notation and describe some prerequisite results. In Section~\ref{sec:critical}, we give some basic properties of a hypothetical critical counterexample to Theorem~\ref{th:newtheorem}. We prove the main technical bound in Section~\ref{sec:engine} and then wrap up the proof of Theorem~\ref{th:newtheorem} in Section~\ref{sec:proof}.

\section{Notation and preliminaries}\label{sec:preliminaries}

Here we introduce some terminology that we use throughout.
We often call $G_1$ the {\em blue} graph and $G_2$ the {\em red} graph.
We treat the injective vertex mappings as labellings of the vertices from $1$ to $n$.
However, rather than saying ``the vertex in $G_1$ (or $G_2$) corresponding to the label $i$'', we often only say ``vertex $i$'', since this should never cause any confusion.
Our proofs rely on accurately specifying the neighbourhood structure as viewed from a particular vertex. Let $i\in [n]$.
The {\em blue neighbourhood $N_1(i)$ of $i$} is the set $\{j \,\mid\, ij \in E(G_1)\}$ and the {\em blue degree $\deg_1(i)$ of $i$} is $|N_1(i)|$.
The {\em red neighbourhood $N_2(i)$} and {\em red degree $N_2(i)$} are defined analogously.
For $j\in[n]$, a {\em red--blue-link (or $2$--$1$-link) from $i$ to $j$} is a vertex $i'$ such that $ii' \in E(G_2)$ and $i'j \in E(G_1)$.
The {\em red--blue-neighbourhood $N_1(N_2(i))$ of $i$} is the set $\{j \,\mid\, \exists\text{ red--blue-link from $i$ to $j$}\}$.
A {\em blue--red-link (or $1$--$2$-link)} and the {\em blue--red-neighbourhood} $N_2(N_1(i))$ are defined analogously.

In search of a certificate that $G_1$ and $G_2$ pack, without loss of generality, we keep the vertex labelling of the blue graph $G_1$ fixed, and permute only the labels in the red graph $G_2$. This can be thought of as ``moving'' the red graph above a fixed ground set $[n]$.
In particular, we seek to avoid the situation that there are $i,j \in [n]$ for which $ij$ is an edge in both $G_1$ and $G_2$ --- in this situation, we call $ij$ a {\em purple} edge induced by the labellings of $G_1$ and $G_2$.
So $G_1$ and $G_2$ pack if and only if they admit a pair of vertex labellings that induces no purple edge.
In our search, we make small cyclic sub-permutations of the labels (of $G_2$), which are referred to as follows.
For $i_0,\dots,i_{\ell-1} \in [n]$, a {\em $(i_0,\dots,i_{\ell-1})$-swap} is a relabelling of $G_2$ so that for each $k \in \{0,\dots,\ell-1\}$ the vertex labelled $i_k$ is re-assigned the label $i_{k+1\bmod \ell}$. In fact, we shall only require swaps having $\ell \in \{1,2\}$.
The following observation describes when a swap could be helpful in the search for a packing certificate.
This is identical to Lemma~1 in~\cite{KKY08}.

\begin{Lemma}\label{lem:swap}
Fix a pair of labellings of $G_1$ and $G_2$ from $[n]$ and let $u_0,\dots,u_{\ell-1} \in [n]$. For every $k, k' \in \{0,\dots,\ell-1\}$, suppose that there is no red--blue-link from $u_k$ to $u_{k+1\bmod \ell}$ and that, if $u_k u_{k'} \in E(G_2)$, then $u_{(k+1 \bmod \ell)} u_{(k'+1 \bmod \ell)} \notin E(G_1)$. Then there is no purple edge incident to any of $u_0,\dots,u_{\ell-1}$ after a $(u_0,\dots,u_{\ell-1})$-swap. \qed
\end{Lemma}

We have already mentioned a `near packing' result of Eaton~\cite{Eaton00} which states that two graphs satisfying the BEC condition admit a pair of labellings such that the purple graph has maximum degree at most $1$.
%
Eaton in fact proved that for two such graphs, if there is a pair of labellings for which the purple graph has some vertex of degree larger than $1$, then there exist $i,j\in[n]$ such that an $(i,j)$-swap yields a pair of labellings with fewer purple edges.
%
The following version of Eaton's result will be of use to us.
\begin{Lemma}[Eaton~\cite{Eaton00}]\label{cor:eaton}
If $G_1$ and $G_2$ satisfy the BEC condition, then for any pair of labellings of $G_1$ and $G_2$ with fewest purple edges the graph induced by the purple edges has maximum degree at most $1$.
\end{Lemma}

We make use of the following corollary to a lemma of Corr\'adi~\cite{Cor69}.
\begin{Lemma}[Corr\'adi~\cite{Cor69}]\label{CorCor}
Let $A_1,\ldots, A_N$ be subsets of a finite set $X$ all of cardinality at least $k$. If there is some integer $t$ such that $k^2 > (t-1) |X|$ and $|A_i \cap A_j| \le t-1$ for all $i \neq j$, then
\begin{align*}
N \le  |X| \frac{k- (t-1)}{k^2-(t-1) |X|}.
\end{align*}
\end{Lemma}

\section{A hypothetical critical counterexample}\label{sec:critical}

We begin the proof of Theorem~\ref{th:newtheorem} in this section and continue it in the next two sections. Our proof is by contradiction. This section is devoted to describing the basic properties of a hypothetical counterexample, one that is critical in a sense we next make precise.

Suppose Theorem~\ref{th:newtheorem} is false. Then there must exist a counterexample, that is, a pair $(G_1,G_2)$ of non-packable graphs on $n$ vertices that satisfy the conditions of the theorem.

By Lemma~\ref{cor:eaton}, there exists a pair $(L_1,L_2)$ of labellings of $G_1$ and $G_2$ from $[n]$ such that the graph induced by the purple edges has maximum degree $1$ {\em and} has the minimum number of purple edges among all pairs of labellings of $G_1$ and $G_2$. From now on, we consider $(G_1,G_2)$ with labellings $(L_1,L_2)$ and we fix an arbitrary edge $uv$ that is purple under $(L_1,L_2)$. We will further describe the neighbourhood structure as viewed from $u$ (or $v$). Estimation of the sizes of subsets in this neighbourhood structure is our main method for deriving upper bounds on $n$ that in turn yield the desired contradiction.

We would like to point out similarities with the approach in~\cite{KKY08} and~\cite{CK16}, where $G_2$ was chosen to be edge-minimal over all pairs $(G_1,G_2)$ of non-packable graphs satisfying the theorem conditions. This led to a hypothetical counterexample with only one purple edge. In the present setting, this approach is infeasible because one of the conditions of Theorem~\ref{th:newtheorem} (namely, that both $\Delta_1$ and $\Delta_2$ are sufficiently large) is not invariant under edge removal in $G_2$. This is what led us to consider the purple-edge-minimal counterexample as described above, where we fix $G_1$ and $G_2$ and only minimise over their labellings. The clear downside is that potentially we are faced with multiple purple edges rather than just one, but since by Lemma~\ref{cor:eaton} these do not interfere, it turns out that we can obtain essentially the same structural properties we could have had if we instead assumed $G_2$ to be edge-minimal. It is possible that this alternative form of the minimal
counterexample approach is useful for proving other results related to
the BEC conjecture.

Note that the condition ``$\Delta_1$ sufficiently large" \textit{is} invariant under removing edges from $G_2$. So the weaker version of Theorem \ref{th:newtheorem} that doesn't include  the condition $\Delta_1 \geq \Delta_2 \geq 27620$ \textit{can} be proved using a $G_2$-edge-minimal counterexample with a unique purple edge, without making use of Eaton's near-packing result.

In order to describe the neighbourhood structure of $u$ and $v$, we need the definition of the following vertex subsets:
\begin{align*}
A(u)
&:= N_2(N_1(u))\setminus (N_1(u)\cup N_2(u)\cup N_1(N_2(u)) ), \\
B(u)
&:= N_1(N_2(u))\setminus (N_1(u)\cup N_2(u)\cup N_2(N_1(u)) ), \\
A^*(u) 
&:=N_2(N_1(u)) \setminus ( N_2(u)\cup N_1(N_2(u)) ),
 \text{ and}\\
B^*(u)
&:= N_1(N_2(u)) \setminus ( N_1(u)\cup N_2(N_1(u)) ).
\end{align*}
These sets are analogously defined for $v$ also, and indeed for any element of $[n]$.

One justification for specifying the above subsets is that the following two claims hold. (These are analogues of Claims~1 and~2 in~\cite{KKY08}.)

\begin{Claim}\label{clm:links1}
For all $w\in [n]\setminus\{v\}$, there is a red--blue-link or a blue--red-link from $u$ to $w$.

For all $w\in [n]\setminus\{u\}$, there is a red--blue-link or a blue--red-link from $v$ to $w$.
\end{Claim}

\begin{proof}
By symmetry, we only need to show the first statement.
If it does not hold, then by Lemma~\ref{lem:swap} a $(u,w)$-swap yields a pair of labellings such that $uv$ is no longer purple and no new purple edges arise. This contradicts the choice of $(L_1,L_2)$.
\end{proof}

\begin{Claim}\label{clm:links2}
For all $a\in A^*(u)$ and $b\in B(u)$, there is a red--blue-link from $a$ to $b$.

For all $b\in B^*(u)$ and $a\in A(u)$, there is a blue--red-link from $b$ to $a$. 
\end{Claim}

\begin{proof}
By symmetry, we only need to show the first statement.
Note that there is at least one purple edge incident to a vertex from $\left\{a,b,u\right\}$, namely $uv$. Since $B(u)\cap N_1(u) = B(u) \cap N_2(u)= \emptyset$ and $A^*(u) \cap N_2(u) = \emptyset$, we have that $bu \notin E(G_1) \cup E(G_2)$ and $ua \notin E(G_2)$. Furthermore, since $A^*(u)\cap N_1(N_2(u))= B(u)\cap N_2(N_1(u))=\emptyset$, there is no red--blue-link from $u$ to $a$ or from $b$ to $u$. Now suppose that there is also no red--blue-link from $a$ to $b$. Then it follows from Lemma~\ref{lem:swap} that there are no purple edges incident to any of $u,a,b$ after a $(u,a,b)$-swap. Since the swap only affects the edges incident to at least one of $\left\{a,b,u\right\}$, this decreases the number of purple edges, contradicting the choice of $(L_1,L_2)$.
\end{proof}

We may assume that $\Delta_1, \Delta_2 \ge 2$ since the BEC conjecture is known for $\Delta_2=1$.
Then the following easy claim shows that neither of $A^*(u)$ and $B^*(u)$ is empty.
\begin{Claim}\label{le:AB1}
$|A^*(u)| \ge \Delta_1-1$ and $|B^*(u)|\ge \Delta_2-1$. And so $|A^*(u)|,|B^*(u)|\ge 1$.
\end{Claim}
\begin{proof}
Suppose otherwise.
If $|A^*(u)| \le \Delta_1-2$, note that $[n] \subseteq N_1(N_2(u)) \cup A^*(u) \cup N_2(u)$ by Claim~\ref{clm:links1}, and so
\begin{align*}
n\le |N_1(N_2(u))| + |A^*(u)| + |N_2(u)| \le \Delta_1\Delta_2 + \Delta_1-2 + \Delta_2.
\end{align*}
Symmetrically, if $|B^*(u)| \le \Delta_2-2$, then
\begin{align*}
n&\le |N_2(N_1(u))| + |B^*(u)| + |N_1(u)| \le \Delta_1\Delta_2 + \Delta_2-2 + \Delta_1.
\end{align*}
In either case, we obtain a contradiction to the assumption that $n \ge (\Delta_1+1)(\Delta_2+1)-1$.
\end{proof}

\section{Engine of the proof}\label{sec:engine}

The following technical bound forms the core of the argument. It bounds the intersection of any two mixed second order neighbourhoods in our hypothetical critical counterexample. 
The bound relies on an application of Corr\'adi's lemma (Lemma~\ref{CorCor}).


\begin{Claim}\label{le:postcorradi}
For any integer $t\ge 2$ and distinct $a, b\in[n]$,
\begin{align*}
|N_1(N_2(a)) \cap N_1(N_2(b))|
& \le  \Delta_1 + \Delta_2 + \sqrt{1.37(t-1)} \Delta_2\sqrt{\Delta_2} + \\
& \hspace{50pt} \frac{\sqrt{1.37}}{0.37\sqrt{t-1}} \Delta_1\sqrt{\Delta_2}  + \frac{1}{t} \Delta_1 \Delta_2 \text{ and}\\
|N_2(N_1(a)) \cap N_2(N_1(b))|
& \le  \Delta_1 + \Delta_2 + \sqrt{1.37(t-1)} \Delta_1\sqrt{\Delta_1} + \\
& \hspace{50pt} \frac{\sqrt{1.37}}{0.37\sqrt{t-1}} \Delta_2\sqrt{\Delta_1} + \frac{1}{t} \Delta_1 \Delta_2.
\end{align*}
\end{Claim}
\begin{proof}
By symmetry we only need to prove the first bound. Our approach to this is to partition $N_1(N_2(a)) \cap N_1(N_2(b))$ into a number of subsets, each of which we bound separately. To assist the reader, we have provided a depiction of our partition scheme in Figure~\ref{fig:vertexsets}.

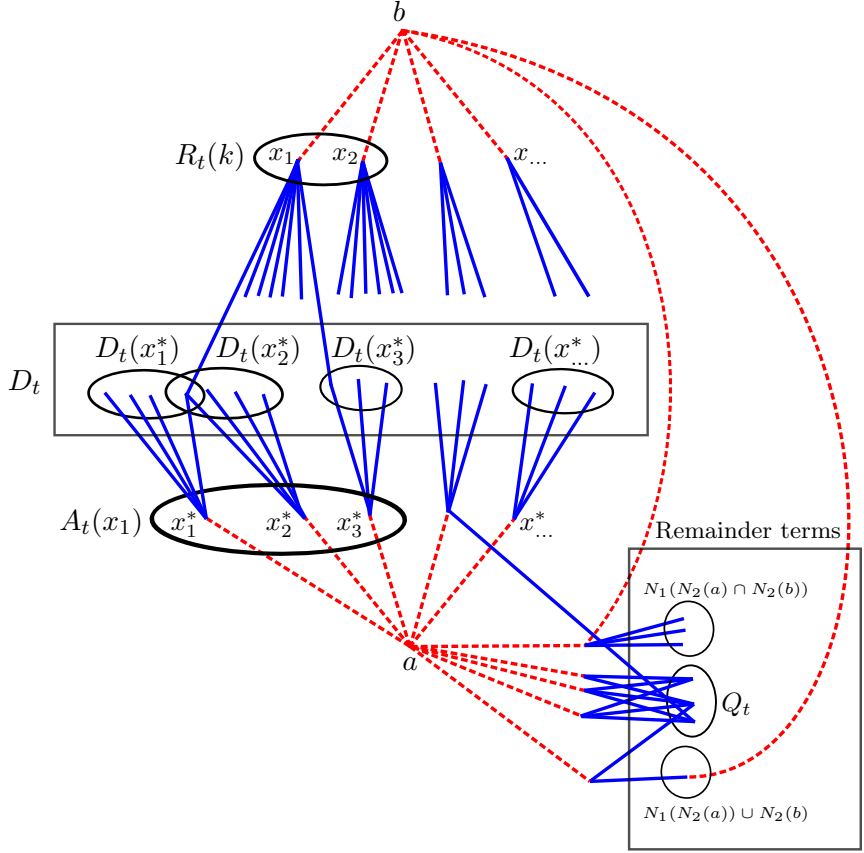
\begin{figure}
 \begin{center}
   \begin{overpic}[width=0.7\textwidth]{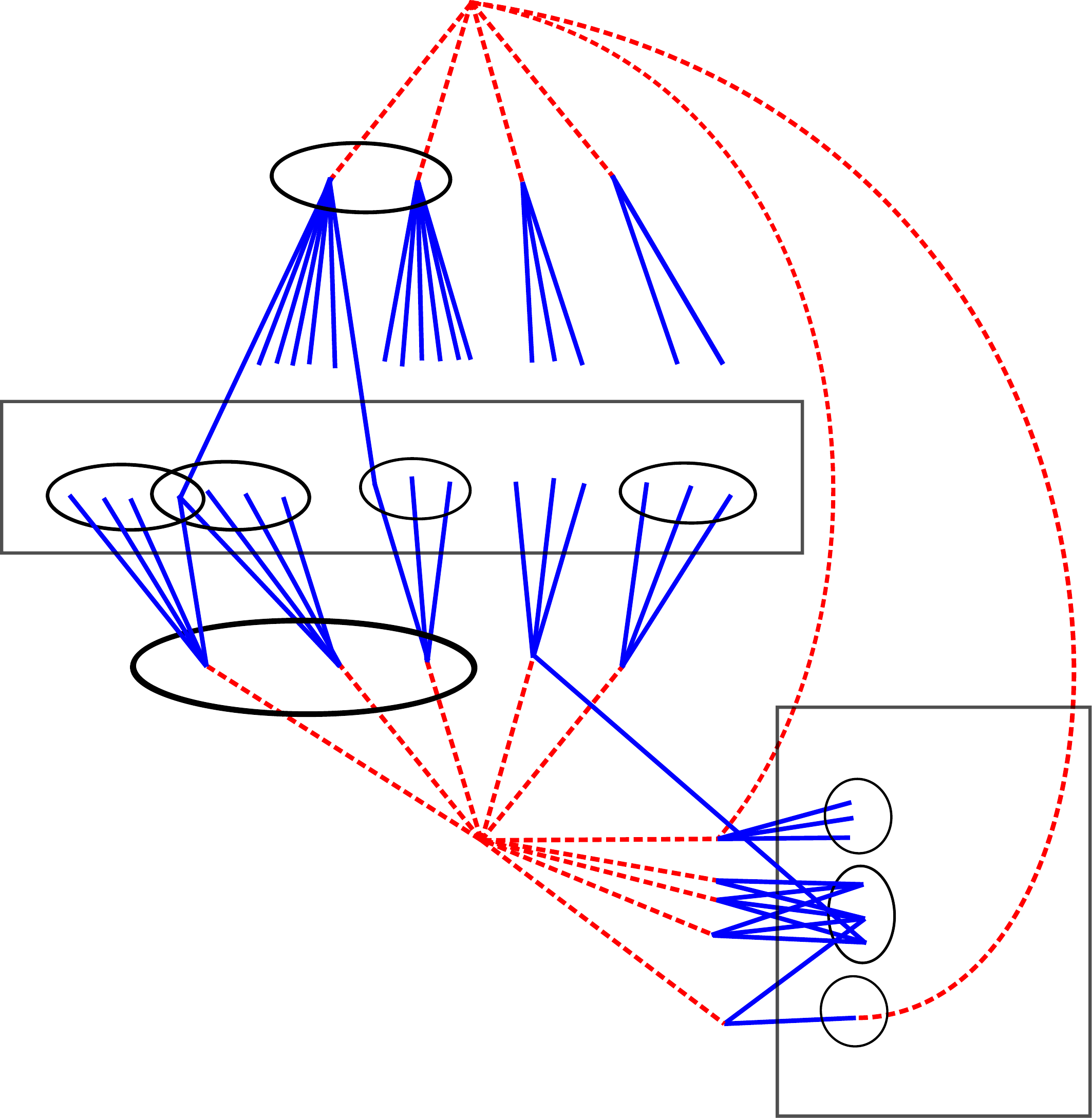}
    \put (41,101){$b$}

    \put (24.8,84){ {\small $x_1$} }
    \put (32.5,84){ {\small $x_2$} }
    \put (54.4,84){ {\small $x_{\dots}$} }
    
    \put (13.5,83.5){ $R_t(k)$ }
     
    \put (5,60){$D_t(x_1^*)$}
    \put (19.6,60){$D_t(x_2^*)$}
    \put (33.5,60){$D_t(x_3^*)$}
    \put (55,60){$D_t(x_{\ldots}^*)$}
    
    \put (-5.5,56.1){$D_t$}
    
	\put (0.75,39){$A_t(x_1)$}      
    
    \put (13,39){ {\small $x_1^*$} }
    \put (24.5,39){ {\small $x_2^*$} }
    \put (33,39){ {\small $x_3^*$} }
    \put (55.1,39){ {\small $x_{\ldots}^*$} }

    \put (41,22){ $a$ }

    \put (70,31.2){ {\tiny $N_1(N_2(a)\cap N_2(b))$} }
    \put (79.5,17){ {\small $Q_t$} }
    \put (70,4){ {\tiny $N_1(N_2(a)) \cup N_2(b)$} } 

  \put (71.5,38) { {\footnotesize Remainder terms }} 

 \end{overpic}
  \caption{A depiction of the vertex sets relevant to the proof of Claim~\ref{le:postcorradi}.\label{fig:vertexsets}}
\end{center}
\end{figure}

Before starting the main argument, we first need to prune the neighbourhood $N_1(N_2(a))$ of three types of relatively small subsets.

\begin{itemize}
\item First, since $G_2$ is $C_4$-free, $|N_2(a) \cap N_2(b)| \le 1$, so 
\begin{align}\label{eq:sep1}
|N_1(N_2(a)\cap N_2(b))|\le \Delta_1.
\end{align} 
Thus we can restrict our attention to $|N_1(N_2(a) \setminus N_2(b)) \cap N_1(N_2(b))|$. The reason for this technical reduction is so that we can work with the \textit{disjoint} sets $N_2(a) \setminus N_2(b)$ and $N_2(b)$.

\item 
Second, define
\begin{align*}
Q_t:= \{ y \in N_1(N_2(a)) \mid |N_1(y)\cap N_2(a)|\ge t \}.
\end{align*}
So $Q_t$ is the set of vertices in $N_1(N_2(a))$ that are in the blue neighbourhoods of at least $t$ different red neighbours of $a$.
We estimate $|Q_t|$ separately, because its elements facilitate a large amount of overlap among the blue neighbourhoods of (at most $t$ different) vertices in $N_2(a)$, while still not violating the absence of large cycles. By an overcounting argument, 
\begin{align} \label{eq:sep2}
|Q_t| \le \sum_{x \in N_2(a)} \sum_{y \in N_1(x)} \frac{\mathbbm{1}_{ \left\{ y \in Q_t  \right\} }}{t} \le \frac{1}{t} \sum_{x \in N_2(a)} \sum_{y \in N_1(x)} 1 \le \frac{\Delta_1\Delta_2}{t}.
\end{align}

\item
Third, we estimate $|N_1(N_2(a)) \cap N_2(b)|$ separately, because later we wish to be able to assume that there are no blue edges between $N_2(a)$ and $N_2(b)$. We have that
\begin{align} \label{eq:sep3}
|N_1(N_2(a))  \cap N_2(b)| \le |N_2(b)| \le \Delta_2. 
\end{align} 
\end{itemize}
 
Having established the estimates~\eqref{eq:sep1},~\eqref{eq:sep2} and~\eqref{eq:sep3} separately, we are left with estimating $|N_1(N_2(b)) \cap  (N_1(N_2(a)\setminus N_2(b)) \setminus (Q_t \cup N_2(b) ))|$, and we do so with Lemma~\ref{CorCor}.
 
For brevity, define $D_t := N_1(N_2(a)\setminus N_2(b)) \setminus (Q_t \cup N_2(b))$ and $D_t(x^*):= N_1(x^*) \setminus (Q_t \cup N_2(b))$ for any vertex $x^*\in N_2(a)\setminus N_2(b)$. Note that $D_t = \bigcup_{x^* \in N_2(a)\setminus N_2(b)} D_t(x^*)$ and our goal now is to bound $|N_1(N_2(b)) \cap D_t|$.
 
Define $k:= \sqrt{1.37(t-1)\Delta_2} $ and let
\begin{align*}
R_t(k):= \left\{ x \in N_2(b) \mid |N_1(x) \cap D_t| > k \right\}.
\end{align*}
So $R_t(k)$ is the set of red neighbours of $b$ that each have `large' blue neighbourhoods intersecting $D_t$. We want to show that $|R_t(k)|$ is small, so without loss of generality we may assume that $k$ is small enough to ensure that $R_t(k) \neq \emptyset$.

For each $x \in N_2(b)$, define the set
\begin{align*}
A_t(x):= \left\{x^* \in N_2(a)\setminus N_2(b) \mid N_1(x) \cap D_t(x^*) \neq \emptyset \right\}.
\end{align*}
For the moment, let us assume that we have established the following two properties:
%
\begin{align}
\label{eq:A_size}
|A_t(x)|
&> k &\text{ for all } x \in R_t(k); \\
%
%
\label{eq:A_overlap}
|A_t(x_1) \cap A_t(x_2)| 
&\le t-1 &\text{ for all distinct } x_1,x_2 \in N_2(b).
\end{align}
We prove these two properties later, but let us first show how from these both a bound on $|R_t(k)|$ and then the desired result follow.

Note that we have chosen $k$ such that $k^2 = 1.37(t-1) \Delta_2 > (t-1) |N_2(a) \setminus N_2(b)|$.
By this choice and the inequalities in~\eqref{eq:A_size} and~\eqref{eq:A_overlap}, we may apply Lemma~\ref{CorCor} with $N= |R_t(k)|$, $X= N_2(a)\setminus N_2(b)$, the parameters $t$ and $k$, and the collection $\left( A_t(x) \right)_{x \in R_t(k)}$ of subsets of $X$, yielding the following bound:
\begin{align*}
|R_t(k)| &\le |N_2(a)\setminus N_2(b)| \cdot \frac{k -(t-1)}{k^2-(t-1) |N_2(a)\setminus N_2(b)|} \\
&\le \Delta_2 \cdot \frac{ \sqrt{1.37(t-1)\Delta_2}}{1.37(t-1) \Delta_2 -(t-1) \Delta_2}
= \frac{\sqrt{1.37}}{0.37}\sqrt{\frac{\Delta_2}{t-1}}.
\end{align*} 


We can then bound the main term as follows:
\begin{align}
|N_1(N_2(b)) \cap D_t| & \le  |\{ x \in N_2(b) \mid |N_1(x)  \cap D_t| \le k \}| \cdot k + |R_t(k)| \Delta_1\nonumber\\
& \le \Delta_2 k + |R_t(k)| \Delta_1
%
\label{eq:mainterm}
\le \Delta_2 \sqrt{1.37(t-1) \Delta_2}  + \frac{\sqrt{1.37}}{0.37}\sqrt{\frac{\Delta_2}{t-1}} \Delta_1\nonumber\\
&= \frac{\sqrt{1.37}}{0.37\sqrt{t-1}} \Delta_1\sqrt{\Delta_2}+\sqrt{1.37(t-1)} \Delta_2\sqrt{\Delta_2}.
\end{align}

Combining inequalities~\eqref{eq:sep1},~\eqref{eq:sep2},~\eqref{eq:sep3} and~\eqref{eq:mainterm}, we obtain
\begin{align*}
&|N_1(N_2(b)) \cap N_1(N_2(a))|\\
& \le |N_1(N_2(b)) \cap  D_t| + |N_1(N_2(a) \cap N_2(b))|+
|N_1(N_2(b)) \cap Q_t| + |N_1(N_2(b)) \cap N_2(b)| \\
&\le \frac{\sqrt{1.37}}{0.37\sqrt{t-1}} \Delta_1\sqrt{\Delta_2}+\sqrt{1.37(t-1)} \Delta_2\sqrt{\Delta_2} +
\Delta_1 + \frac{1}{t} \Delta_1 \Delta_2 + \Delta_2,
\end{align*} 
which is the desired result.

So to complete the proof, it only remains to show the two properties~\eqref{eq:A_size} and~\eqref{eq:A_overlap}.

For~\eqref{eq:A_size}, since $G_1$ has no $4$-cycle, it holds that $|N_1(x) \cap D_t(x^*)|\le 1$  for each $x \in N_2(b)$ and $x^* \in N_2(a)\setminus N_2(b)$. So for a fixed $x\in R_t(k) \subseteq N_2(b)$, each $x^* \in N_2(a) \setminus N_2(b)$ contributes at most $1$ to $|N_1(x) \cap D_t|$. This proves~\eqref{eq:A_size}.

To prove~\eqref{eq:A_overlap}, suppose for a contradiction that there exist distinct $x_1,x_2 \in N_2(b)$ such that $|A_t(x_1) \cap A_t(x_2)| \ge t$. Then there are at least $t$ different vertices $x_1^*,\ldots, x_t^* \in N_2(a) \setminus N_2(b)$, and there exist vertices $y_{11} \in D_t(x_1^*)\cap N_1(x_1)$, \dots, $y_{t1} \in D_t(x_t^*)\cap N_1(x_1)$ as well as vertices $y_{12} \in D_t(x_1^*)\cap N_1(x_2)$, \ldots, $y_{t2} \in D_t(x_t^*)\cap N_1(x_2)$. Due to the separate estimate~\eqref{eq:sep2}, we were allowed to exclude elements of the set $Q_t$ in our choice of the sets $D_t(\cdot)$, and so the vertices $y_{11}, \ldots y_{t1}$ are not all equal. Recall that we assumed $t\ge 2$. Without loss of generality, we may assume that $y_{11} \ne y_{21}$. Note though that some of the vertices $y_{11}$, $y_{21}$, $y_{12}$, $y_{22}$ may well be equal. Due to the separate estimate~\eqref{eq:sep3}, we were also allowed to exclude elements of $N_2(b)$ in our choice of $D_t(\cdot)$, and so $x_1x_1^*$, $x_1x_2^*$, $x_2x_1^*$, $x_2x_2^*$ are not blue edges. Therefore $\left\{x_1,x_2,x_1^*,x_2^* \right\} \cap \left\{y_{11},y_{12},y_{21},y_{22} \right\} = \emptyset$.

It can be shown that the induced subgraph $G_1[ \{x_1,x_2,x_1^*,x_2^*,y_{11}, y_{12},y_{21}, y_{22} \}]$ contains a $4$-, $6$- or $8$-cycle, which is a contradiction. 
To wit, the case analysis proceeds as follows. See Figure~\ref{fig:cyclecases} for a pictorial synopsis.
Since $y_{11}\ne y_{21}$, there are four cases for the possible coincidences among $y_{11}$, $y_{21}$, $y_{12}$, $y_{22}$:
\begin{enumerate}
\item
{\em The vertices are all distinct.} Then $y_{11}x_1y_{21}x_2^*y_{22}x_2y_{12}x_1^*$ is a blue $8$-cycle.
\item
{\em Exactly one pair of the vertices coincides.} Since $y_{11}\ne y_{21}$, there are five subcases: $y_{11}=y_{12}$, $y_{11}=y_{22}$, $y_{12}=y_{21}$, $y_{12}=y_{22}$, or $y_{21}=y_{22}$. We can consider each subcase individually (as in Figure~\ref{fig:cyclecases}), or we can also notice some symmetries by  a relabelling of the vertices $x_1,x_2,x_1^*,x_2^*,y_{11}, y_{12},y_{21}, y_{22}$.
The three subcases $y_{11}=y_{12}$, $y_{12}=y_{22}$ and  $y_{21}=y_{22}$ are symmetric, and in the first of these subcases $y_{11}x_1y_{21}x_2^*y_{22}x_2$ is a blue $6$-cycle. The two remaining subcases $y_{11}=y_{22}$ and $y_{12}=y_{21}$ are symmetric, and in the first of these $y_{11}x_1y_{21}x_2^*$ is a blue $4$-cycle.
\item
{\em A triple of the vertices coincides.} Since $y_{11}\ne y_{21}$, there are two subcases: $y_{12}=y_{21}=y_{22}$ or $y_{11}=y_{12}=y_{22}$. In the first of these $y_{11}x_1y_{12}x_1^*$ is a blue $4$-cycle, while in the second $y_{11}x_1y_{21}x_2^*$ is a blue $4$-cycle.
\item
{\em Two pairs of the vertices coincide.} Since $y_{11}\ne y_{21}$, there are two subcases: $y_{11}=y_{12},y_{21}=y_{22}$ or  $y_{11}=y_{22},y_{12}=y_{21}$. In both of these $y_{11}x_1y_{21}x_2$ is a blue $4$-cycle. \qedhere
\end{enumerate}
\end{proof}

\begin{figure}
 \begin{center}
	   \begin{overpic}[width=1\textwidth]{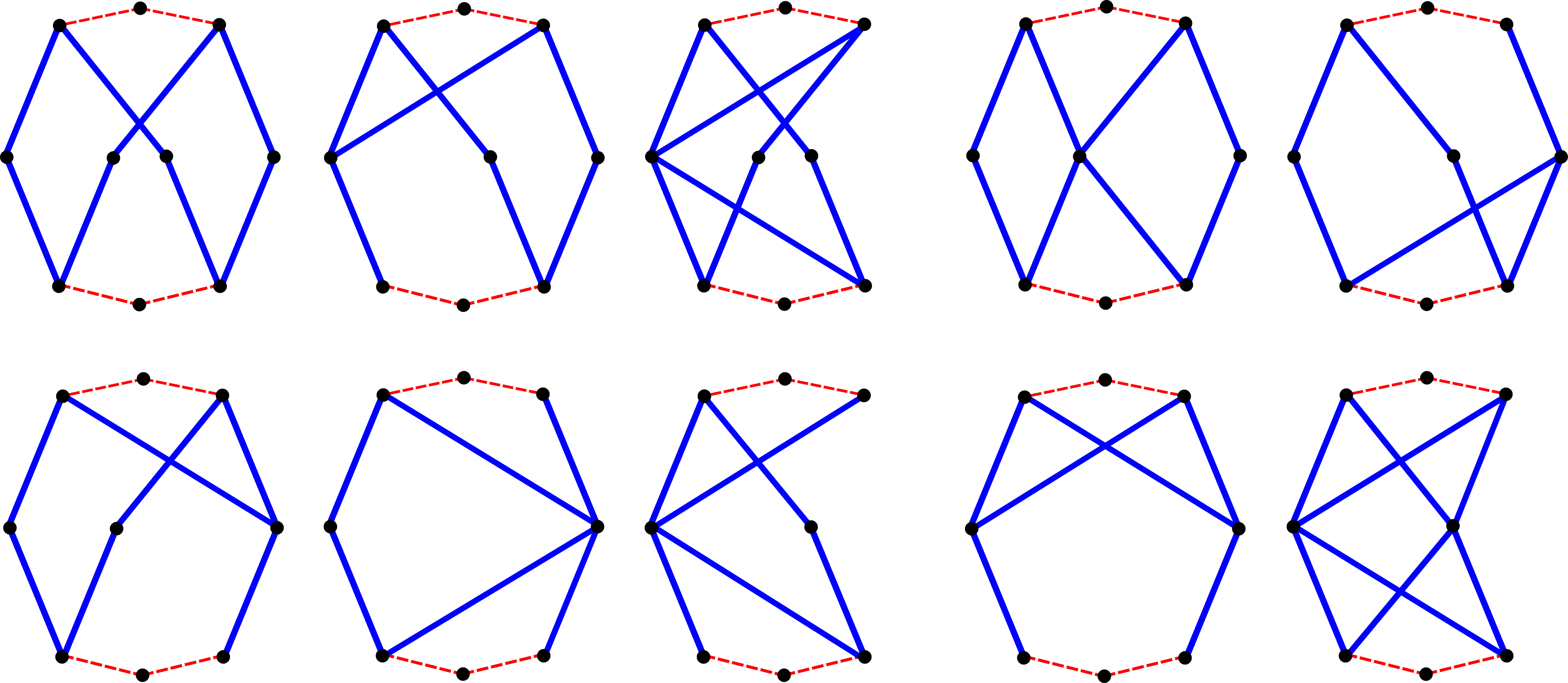}
    \put (7.5,44){ $b$}
    \put (0,42.5){ $x_1$}
    \put (13.5,42.5){ $x_2$}    
    \put (0.5,32.5){  {\scriptsize $y_{11}$}  }
    \put (3.8,34.4){  {\scriptsize $y_{12}$}  }   
    \put (9.9,34.4){  {\scriptsize $y_{21}$}  }
    \put (13.2,32.5){  {\scriptsize $y_{22}$}  }
    \put (0,23){ $x_1^*$}
    \put (13.5,23){ $x_2^*$}  
    \put (7.5,22){ $a$}  
 \end{overpic}
  \caption{The cases analysed in Claim~\ref{le:postcorradi}. We know $y_{11} \neq y_{21}$, but some of the vertices $y_{11}$, $y_{12}$, $y_{21}$, $y_{22}$ may coincide. As shown here, in each case there is a blue $4$-, $6$- or $8$-cycle. In reading order, the depicted cases are: (a) all are distinct, (b)--(f) exactly one pair of vertices coincides, (g)--(h) a triple of vertices coincides, (i)--(j) two pairs of vertices coincide. \label{fig:cyclecases}}
\end{center}
\end{figure}
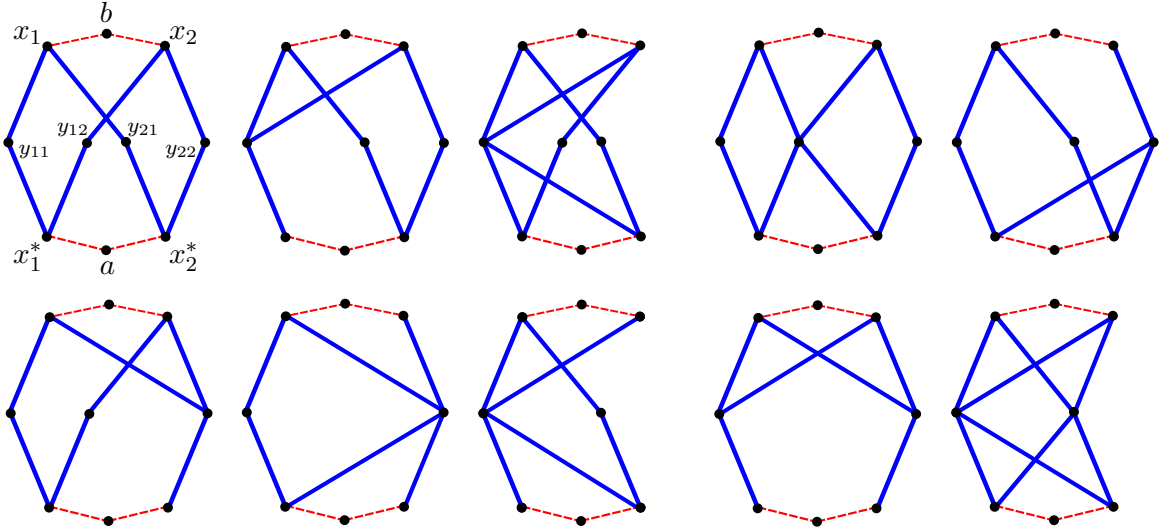

We in fact use a weaker but handier version of Claim~\ref{le:postcorradi}.
For each $t\ge 2$, define
\begin{align*}
C_t := \frac{\sqrt{1.37}}{0.37\sqrt{t-1}} + \sqrt{1.37(t-1)}.
\end{align*}

\begin{Claim}\label{le:mixedNN}\label{le:AB2}
For each $t\ge 2$, we have that $\Delta_1+\Delta_2+C_t \Delta_1\sqrt{\Delta_1} + \Delta_1\Delta_2/t$ is an upper bound for each of the following quantities:
$|N_1(N_2(u)) \cap N_1(N_2(v))|$, $|N_2(N_1(u)) \cap N_2(N_1(v))|$, $|A(v)|$, $|B(v)|$, $|A(u)|$, $|B(u)|$.
\end{Claim}
\begin{proof}
For the first two quantities, apply Claim~\ref{le:postcorradi} with $a=v$ and $b=u$ and note that $\Delta_1 \ge \Delta_2$, by assumption.

For the last quantity, note first that $|A^*(u)|\ge 1$ by Claim~\ref{le:AB1}. By Claim~\ref{clm:links2}, there exists $a\in A^*(u)$ (not equal to $u$) such that $B(u) \subseteq N_1(N_2(a)) \cap N_1(N_2(u))$.
The bound follows from Claim~\ref{le:postcorradi} with $a$ and $b=u$ and the assumption that $\Delta_1 \ge \Delta_2$.
The proof for the remaining quantities is the same.
\end{proof}

\section{Conclusion of the proof}\label{sec:proof}

We are ready to complete the proof of Theorem~\ref{th:newtheorem}.



We have by Claim~\ref{clm:links1} that
\begin{align*}
[n] &\subseteq N_1(N_2(u)) \cup A^*(u) \cup N_2(u), \\
[n] &\subseteq N_1(N_2(v))  \cup A^*(v) \cup N_2(v), \text{ and}\\
[n] &\subseteq N_2(N_1(v))  \cup B^*(v) \cup N_1(v).
\end{align*}
So it follows (also using the definitions of $A^*(v)$, $A(v)$,  $A^*(u)$,  $B^*(v)$, $A(u)$, $B(v)$) that
\begin{align}\label{eq:nbound}
n
&\le |N_1(N_2(u))|  + |A^*(u)| +  |N_2(u)|\nonumber\\
& \le (|N_1(N_2(u)) \cap N_1(N_2(v))| + |N_1(N_2(u)) \cap A^*(v)| + |N_1(N_2(u)) \cap N_2(v)|)+\nonumber\\
& \hspace{25pt}(|A^*(u) \cap N_2(N_1(v))| + |A^*(u) \cap B^*(v)| + |A^*(u) \cap N_1(v)|)   +   |N_2(u)|\nonumber\\
& \le (|N_1(N_2(u)) \cap N_1(N_2(v))| + | A(v)| + |N_1(v)|+ |N_2(v)|)+\nonumber\\
& \hspace{25pt}(|N_2(N_1(u)) \cap N_2(N_1(v))|  + |A(u) \cap B(v)|+ |N_1(u)| + |N_2(v)| + |N_1(v)|)    +    |N_2(u)|\nonumber\\
& \le |N_1(N_2(u)) \cap N_1(N_2(v))| +|N_2(N_1(u)) \cap N_2(N_1(v))| +| A(v)|+ |B(v)|+ 3 (\Delta_1 + \Delta_2)\nonumber\\
& \le 4 C_t  \Delta_1\sqrt{\Delta_1} + 4 \Delta_1\Delta_2/t+ 7(\Delta_1 + \Delta_2),
\end{align}
where to derive the last line we applied Claim~\ref{le:mixedNN} for some $t \ge 2$ to be specified later.
Routine arithmetic manipulations show that, if 
\begin{align}\label{cond1}
\sqrt{\Delta_1} < \frac{t-4}{4 t C_t}\Delta_2- \frac{3}{C_t} = \frac{1}{4 t C_t}((t-4)\Delta_2- 12t),
\end{align}
then~\eqref{eq:nbound} is strictly less than $(\Delta_1+1)(\Delta_2+1)-(1+6(\Delta_1-\Delta_2)) \le (\Delta_1+1)(\Delta_2+1)-1$, contradicting our assumption on $n$.
Moreover, by Theorem~\ref{th:oldtheorem}, if 
\begin{align}\label{cond2}
\Delta_1 \ge 34\Delta_2,
\end{align}
then $G_1$ and $G_2$ pack, also a contradiction.
Thus neither of~\eqref{cond1} and~\eqref{cond2} holds, and so
\begin{align*}
\frac{136 t C_t}{t-4} \left( \sqrt{\Delta_1} + \frac{3}{C_t} \right) \ge 34 \Delta_2 > \Delta_1 \ge \frac{1}{16 t^2 C_t^2} \left( (t-4)\Delta_2- 12t \right)^2.
\end{align*}
This in turn yields the following two quadratic polynomial inequalities:
\begin{align*}
(t-4)^2 \Delta_2^2 - (544 t^2 C_t^2 + 24 t) \Delta_2 + 144 t^2 < 0 &\text{ and}\\
(t-4) \Delta_1 - 136 t C_t \sqrt{\Delta_1} - 408 t < 0 &.
\end{align*}
A good choice of $t$ turns out to be $t=15$. Substituting this (and the formula for $C_t$) into the above two inequalities yields that $\Delta_2 < 27620$ and $\Delta_1 < 940060$. This contradicts our assumptions on $\Delta_1$ and $\Delta_2$, and this completes the proof.
\qed

\bibliographystyle{abbrv}
\bibliography{packingplus}

\end{document}